\documentclass[12pt, reqno]{amsart}

\usepackage{amsmath,amsfonts, amsthm, amssymb, color, cite}
\usepackage{color}
\usepackage{fancyhdr}
\usepackage{graphicx}
\usepackage{epstopdf}
\usepackage{lastpage}
\newtheorem{thm}{Theorem}[section]
\newtheorem{lma}{Lemma}[section]
\newtheorem{pro}{Proposition}[section]

\theoremstyle{definition}

\theoremstyle{remark}

\numberwithin{equation}{section}
\allowdisplaybreaks

\newcommand{\R}{{\mathbb R}}

\def\f{\frac}

\def\hf1{^\f{1}{1-\xi^2}}

\def\be{\begin{equation}}
\def\en{\end{equation}}
\def\bs{\begin{split}}
\def\es{\end{split}}
\def\ba{\begin{align}}
\def\ea{\end{align}}

\author[Lin Chang]{Lin Chang}
\address{School of Mathematics and Physics, Handan University, Handan 056006, China}
\email{ changlin23@buaa.edu.cn}

\author[Duo Liu]{ Duo Liu}
\address{School of Mathematics and Physics, Handan University, Handan 056006, China}
\email{ liuduo8610@163.com}

\author[W. Zhang]{Weiqiang  Zhang}
\address{Xi'an ShaanGu Power Co.,Ltd}
\email{zdhb1@shaangu.com}
\renewcommand{\fancyhead}{}
\title {S\MakeLowercase{tability} \MakeLowercase{of}     r\MakeLowercase{arefaction} w\MakeLowercase{aves}     u\MakeLowercase{nder} p\MakeLowercase{eriodic}
p\MakeLowercase{erturbation}     \MakeLowercase{for} a r\MakeLowercase{ate}-t\MakeLowercase{ype} v\MakeLowercase{iscoelastic} s\MakeLowercase{ystem} }
\date{\today}
\begin{document}
\begin{abstract}
	In this paper, a rarefaction wave under space-periodic perturbation for the  $3\times3$ rate-type viscoelastic system is considered. It is shown that if the initial perturbation around the rarefaction wave is suitably small, then the solution of the rate-type viscoelastic system tends to the rarefaction wave.
	The stability of solutions under periodic perturbations is an interesting and important problem since the perturbation keeps oscillating at the far fields. That is, the perturbation is not integral in space. The key of proof is to construct a suitable ansatz carrying the same oscillation as the solution. Then we can find cancellations between solutions and ansatz such that the perturbation belongs to some Sobolev space. The nonlinear stability can be obtained by the weighted energy method.
\end{abstract}
\maketitle

\section{Introduction }
In this paper, we study the following rate-type viscoelastic system, which reads
\begin{equation}
	\left\{\begin{array}{l}{v_{t}-u_{x}=0}, \\
		{ u_{t}+p_{x}=0}, \\
		{(p+E v)_{t}=\frac{p_{R}(v)-p}{\tau}}.
	\end{array}\right.
	\label{1.1}\end{equation}
where $v$ and$(-p)$ denote strain and stress, $u$ is related to the particle velocity, $E$ is a positive constant called the dynamic Young's modulus, $\tau>0$ is a relaxation time, we  assume  $\tau=1$    in  this paper. $p_R(v)$  is a equilibrium value for $p$.

In this paper, we are concerned about a Cauchy problem for (\ref{1.1}) with the initial data
\begin{equation}\label{ic}
	(v,u,p)(x,0) = (v_0,u_0,p_{0})(x), \quad x\in\R,
\end{equation}
which satisfies
\begin{equation}\label{end-behavior}
	(v_0,u_0,p_{0})(x) \rightarrow \begin{cases}
		( \bar{v}_{l},\bar{u}_{l},\bar{p}_{l})+\left(\phi_{0l}, \psi_{0l},p_R( \bar{u}_l+\psi_{0l})-p_R(\bar{u}_l) \right)(x) & \quad \text{as } x\rightarrow -\infty,\\
		(\bar{v}_{r},\bar{u}_{r},\bar{p}_{r})+(\phi_{0r}, \psi_{0r},p_R( \bar{u}_r+\psi_{0r})-p_R(\bar{u}_r))(x) & \quad \text{as } x\rightarrow +\infty,
	\end{cases}
\end{equation}
where $ \bar{v}_{l}>0, \bar{v}_{r}>0, \bar{u}_{l} $ and $ \bar{u}_{r} $ are constants, $ (\phi_{0l}, \psi_{0l}) \in \R^2 $ and $ (\phi_{0r}, \psi_{0r}) \in \R^2 $ are periodic functions with periods $ \pi_l>0 $ and $ \pi_r>0, $ respectively.

This system was proposed in \cite{SI1990} to introduce a relaxation approximation to the following  system
\begin{equation}\label{1.2}
	\left\{\begin{array}{l}{v_{t}-u_{x}=0}, \\ {u_{t}+p_{R}(v)_{x}=0.}\end{array}\right.
\end{equation}
The type of system (\ref{1.2}) strictly depends on the property of $p_{R}(v) .$ Suliciu\cite{SI1990} got the semilinear hyperbolic system, which is independent of the property of $p_{R}(v),$  by introducing a visco-elastic regularisation which allows a pressure relaxation over time.     $\mathrm{Li}$ \cite{6}  studied the case of $\bar{u}_{l}>\bar{u}_{r}$ and obtained the  asymptotic behavior of the solution as time goes to infinity. So we discuss the case of $\bar{u}_{l}<\bar{u}_{r}$in this paper.


There are many results on the nonlinear stability of the elementary waves. Under certain restrictions on given initial data, Hsiao and Luo \cite{HT} studied the nonlinear stability of the corresponding shock profile for the system (\ref{1.1}), while Luo and Serre  \cite{10} investigated  the linear stability of the shock profile. Moreover, Hsiao and $\mathrm{Pan} $ \cite{HP1999}   discussed the nonlinear stability of the corresponding rarefaction wave.

The large-time behavior of elementary waves has always been a mathematical problem that has attracted much attention, and there have been many excellent results. For the case of shock wave, we refer to \cite{mn1985,hm2009} for the stability of a single shock and a composition of two shocks. For the case of rarefaction wave, we refer to \cite{MN1986}. For the case of contact discontinuity, we refer to \cite{HMX,LWX3} for the stability and optimal decay rate of contact discontinuity. Furthermore, we refer to \cite{HX,LWX2} for the decay rate of shock waves. There are also many interesting results of pointwise estimates for elementary wave patterns by the method of Green's function, we refer to \cite{Liu1,Liu2,Liu3} and the reference therein. For the multi-dimensional periodic perturbations, we refer to \cite{HXXY,HXY2022,HLWX}. In \cite{HLWX}, the exponential decay rate of non-zero modes is also obtained.


We aim to prove that the   rarefaction wave is stable for the Cauchy problem \eqref{1.1}-\eqref{end-behavior}. Roughly speaking, the solution not only exists globally but also tends to be a rarefaction wave  as time goes to infinity.  The precise statements of the main results are given in Theorem \ref{theorem201} in Section 2.
We outline the strategy as follows. We apply the energy method to study the stability of the rarefaction wave  $(V^{r}, U^{r}, P_R(V^{r}))$, in which the energy of the perturbation $(v, u, p)-(V^{r}, U^{r}, P_{R}(V^{r}))$, namely,  $ (v, u, p)-(V^{r}, U^{r}, P_{R}(V^{r}))$, ``should'' belong to some Sobolev spaces like  $H^1(\mathbb{R})$.  However, the method above can not be applicable directly in this paper since the perturbation  oscillates at the far field and hence does not belong to any $L^p$ space for $p\ge 1$. Motivated by \cite{HXY2022,LWX},  we introduce a suitable ansatz $(V,U,P_R(V)(x,t)$, which has the same oscillations as the solution $(v,u,p)(x,t)$ at the far field, so that the perturbation belongs to some Sobolev spaces and the energy method is still available.

Now, we show the method to construct the  ansatz. Taking $V$ as an example,   $V$ $= v_{l}(x,t)g  (x,t) + v _ { r } (x, t ) [ 1 - g (x,t) ]$, where  $(v_{ l },u_{ l })$ is a periodic solution of \eqref{1.2} with the initial data $(\bar{v}_{l}+\phi_{0 l}(x),\bar{u}_{l}+\psi_{0 l}(x))$ in \eqref{1.2} and is expected to have the same oscillation as $v$ near $x= -\infty$, and $g$ is a weighted function satisfying $\lim_{x\rightarrow -\infty} g  (x,t)=0, \lim_{x \rightarrow +\infty} g  (x,t)=1 .$ Similarly $v_r$ is expected to have the same oscillation as $v$ near $x=\infty$. Thus $v(x,t)-V(x,t)$ is expected to be integrable. Thus the perturbation $v-V$ could belong to $L^2$.

The rest of the paper will be arranged as follows. In Section \ref{section2}, a suitable ansatz is constructed and the main results are stated. In Section \ref{section3},  the stability problem is reformulated to a perturbation equation around the ansatz.  In Section \ref{section4},      a priori estimates are established. In Section \ref{section5},   the main results   are  proved.
\noindent
\color{black}

\section{Preliminaries and  Main Results  }\label{section2}
Consider the following Riemann problem:
\begin{equation}
	\left\{ \begin{array}{ll}
		v_{t}-u_{x}=0,\\
		u_t+(P_R(v))_x=0,
	\end{array} \right.
	\label{2.1}\end{equation}
with the initial data
\begin{equation}
	(v(x,0),u(x,0))=(v_{0}^{r}(x),u_{0}^{r}(x)),
	\label{2.2}\end{equation}
where
\begin{align}\label{2.3}
	(v_{0}^{r}(x),u_{0}^{r}(x))=\left\{\begin{array}{l }  (\bar{v}_{l},\bar{u}_{l})  ; x<0,\\
		(\bar{v}_{r},\bar{u}_{r}) ;  x>0.
	\end{array}\right.
\end{align}
We give the following hypotheses:
\begin{align}\label{2.4}
	\begin{split}
		&\left\{\begin{array}{l}
			p'_R(v)<-a_1<0 ;\\
			0<  p''_R(v)< a_2;\\
			|p'_R(v)|<E   ;\\
			p_R(v), p'_R(v), p''_R(v), p'''_R(v)    \text{ bounded.}
		\end{array}\right.\\
	\end{split}&
\end{align}
Here $a_1$, $a_2$ are two constant and $v\in[c_1,d_1]$, where $-\infty<c_1<\bar{v}_{l},\bar{v}_{r}<d_1<\infty. $ Under $ (\ref{2.4})_1$-$ (\ref{2.4})_2$ , the eigenvalues of (\ref{2.1}) are
\begin{align*}
	\lambda_{1}=-(-p'_R(v))^{\frac{1}{2}}<0<(-p'_R(v))^{\frac{1}{2}}=\lambda_{2}.
\end{align*}

We construct a smooth approximation   $ U^R(t, x)$ of the Riemann solution $u^{r}(t, x)$. The Riemann problem for typical Burgers equation  is:
\begin{equation}
	\left\{ \begin{array}{ll}
		w_{t}^{r}+w^{r} w_{x}^{r}=0,\\
		w^{r}(0, x)=w_{0}^{r}(x)=\left\{\begin{array}{ll}{\bar{w}_{l} \equiv \lambda_{1}\left( \bar{u}_{l}   \right)<0,} & {x<0}, \\ {\bar{w}_{r} \equiv \lambda_{1}\left( \bar{u}_{r}\right)<0,} & {x>0},\end{array}\right. \\
	\end{array} \right.
	\label{2.6}\end{equation}\\
with $\bar{w}_{l}<\bar{w}_{r}<0$, (\ref{2.6})    has a continuous weak solution of the form  $w^{r}\left(\frac{x}{t}\right)$ given by
\begin{align*}
	w^{r}(\xi)=\left\{\begin{array}{ll}{\bar{w}_{l},} & {\xi \leq \bar{w}_{l}}, \\
		{\xi,} & {\bar{w}_{l} \leq \xi \leq \bar{w}_{r}}, \\
		{\bar{w}_{r},} & {\xi \geq \bar{w}_{r}}.
	\end{array}\right.
\end{align*}
It is easy to see that the unique solution $u^{r}(t, x)$  of the Riemann problem (\ref{2.1}), (\ref{2.2}) is given by
\begin{align*}
	\left\{\begin{array}{l}{u^{r}(x, t)-\bar{u}_{l}=-\int_{\bar{v}_{l}}^{v^{r}(x, t)} \lambda_{1}(v) d v}, \\ {\lambda_{1}\left(v^{r}\right)(x, t)=w^{r}(x, t)}.
	\end{array}\right.
\end{align*}
We approximate    $w^{r}\left(\frac{x}{t}\right)$   by the solution    $ w(t, x)$    of the following problem
\begin{align}
	\left\{\begin{array}{l}{w_{t}+w w_{x}=0} \\ {w(0, x)=w_{0}(x) \equiv \hat{w}+\tilde{w}   \tanh  x,  }\end{array}\right.
	\label{2.9}
\end{align}
where $\hat{w}=\frac{\bar{w}_{r}+\bar{w}_{l}}{2}$  , $\tilde{w}=\frac{\bar{w}_{r}-\bar{w}_{l}}{2}>0$. Using (\ref{2.9}) we can construct $({}{V^r},{}{U^r})$ below
\begin{align*}
	\left\{\begin{array}{l}{{}{U^r}(x, t)-\bar{u}_{l}=-\int_{\bar{v}_{l}}^{{}{ V^r}(x, t)} \lambda_{1}(v) d v}, \\ {\lambda_{1}\left({}{V^r}\right)(x, t)=w(x, t)}.
	\end{array}\right.
\end{align*}

\begin{lma} \label{xingzhi}(\cite{HP1999})
	Under  the condition $(\ref{2.4})$, there exists a smooth function $(V^r(x,t),U^r(x,t))$, which is the smooth approximation of $(u^r, u^r)$ satisfying  the following properties \\
	\begin{enumerate}
		\item
		\begin{eqnarray}\label{m2.11}
			\left\{\begin{array}{l}{V^r_{t}-U^r_{x}=0}, \\ {U^r_{t}+\left(p_{R}(V^r)\right)_{x}=0},
			\end{array}\right.
		\end{eqnarray}
		\item \begin{eqnarray}\label{2.12}
			\lim _{t \rightarrow+\infty} \sup _{x \in \R }\left\{\left|v^{r}(x, t)-V^r(x, t)\right|+\left|u^{r}(x, t)-U^r(x, t)\right|\right\}=0,
		\end{eqnarray}
		\item \begin{eqnarray}\label{2.13}
			\frac{\partial V^r}{\partial t}>0, \forall x \in  \R , t \geq 0
		\end{eqnarray}
		\item $\forall p \in[1,+\infty], \exists c_{p}>0 , s.t.  \forall t \geq 0 ,$
		\begin{equation}\label{2.14}
			\left\|\left(V^r_{x}, U^r_{x}\right)\right\|_{L^{p}} \leq c_{p} \delta^{\frac{1}{p}}(1+t)^{-1+\frac{1}{p}}, \quad\left\|\left(V^r_{x}, U^r_{x}\right)\right\|_{L^{\infty}} \leq c_{\infty} \delta
		\end{equation}where  $\delta $ is the strength of the wave, namely, $ \delta :=\left|\bar{v}_{r}-\bar{v}_{l}\right|+\left|\bar{u}_{r}-\bar{u}_{l}\right| .$
		\item for  $j \geq 2, \forall p \in[1,+\infty), \exists c_{p, j}>0 $, s.t.,  $\forall t \geq 0 $,
		\begin{equation}\label{cccc2.15}
			\left\|\frac{\partial^{j}}{\partial x^{j}}(V^r, U^r)\right\|_{L^{p}} \leq c_{p, j} \delta(1+t)^{-1},
		\end{equation}
		\item $ \exists c>0 $, s.t. \begin{equation}\label{2.15} \left|V^r_{t}\right| \leq c\left|V^r_{x}\right|,\left|U^r_{t}\right| \leq c\left|U^r_{x}\right| ,\end{equation}
	\end{enumerate}
\end{lma}
\subsection{ Suitable   Ansatz}
We construct weighted function below
\begin{equation}\label{g2.11}
	g_{1}(x,t):=\frac{V^r(x,t)-\bar{v}_{l}} {\bar{v}_{r}-\bar{v}_{l}}, \quad g_{2}(x,t):=\frac{U^r(x,t)-\bar{u}_{l}} {\bar{u}_{r}-\bar{u}_{l}}.
\end{equation}
Obviously, it holds that
$$\lim_{x\rightarrow -\infty} g_{j} (x,t)=0, \lim_{x \rightarrow +\infty} g_{j} (x,t)=1; \quad j=1,2.$$
We assume that the function $(v_i,u_i)(x,t) $  is the solution of \eqref{1.1}$_{1,2}$ with the  initial data ($i=l,r$):
\begin{align*}
	\begin{split}
		v_{i0}(x)&:=\bar{u}_{i}+\phi_{0i}(x),\\
		u_{i0}(x)&:= \bar{u}_{i}+\psi_{0i}(x).
	\end{split}&
\end{align*}
\begin{lma}\label{Lem-periodic}(\cite{HY2021})
	Assume that $ (v_0,u_0)(x)\in H^k(0,\pi) $ with $ k\geq 2 $ is periodic with period $ \pi>0 $ and average $ (\bar{v},\bar{u}). $ Then there exists $ \epsilon_0>0 $ such that if
	\begin{equation*}
		\epsilon:=\|{(v_0,u_0)-(\bar{v},\bar{u}) }\|_{H^k(0,\pi)} \leq \epsilon_0,
	\end{equation*}
	the problem \eqref{1.1}$_{1,2}$ with initial data $ (v_0,u_0) $ admits a unique periodic solution $$ (v,u)(x,t) \in C(0,+\infty;H^k(0,\pi)), $$ which has same period and average as $ (v_0,u_0). $ Moreover, it holds that
	\begin{equation}\label{decay-per}
		\|(v,u)-(\bar{v},\bar{u})\|_{H^k(0,\pi)}(t) \leq C \epsilon e^{-\alpha t}, \quad t\geq 0,
	\end{equation}
	where the constants $ C>0 $ and $ \alpha>0 $ are independent of $ \epsilon $ and $ t. $
\end{lma}
Motivated by \cite{HXY2022}, we construct an ansatz below
\begin{align}\label{x2.6}
	\begin{split}
		V(x,t)&:= v_{l}(x,t)g_{1}(x,t)+v_{r}(x,t)[1-g_{1}(x,t)],\\
		U(x,t)&:= u_{l}(x,t)g_{2}(x,t)+u_{r}(x,t)[1-g_{2}(x,t)],\\
		P(x,t)&:= P_R(V).
	\end{split}&
\end{align}
By direly calculation, one obtain that
\begin{equation}\label{2.11}
	\left\{\begin{array}{l}{V_{t}-U_{x}=h_1}, \\ {U_{t}+\left(p_{R}(V)\right)_{x}=h_2},
	\end{array}\right.
\end{equation}
where
\begin{align}\label{x2.11}
	\begin{split}
		h_1=& (v_r-\bar{v}_r+\bar{v}_l-v_l)g_{1t}-(u_r-\bar{u}_r+\bar{u}_l-u_l)g_{2t},\\
		&+(u_r-u_l)_x(g_1-g_2),\\
		h_2= &(v_r-\bar{v}_r+\bar{v}_l-v_l)g_{2x}+[p'_R(V)-p'_R(v_l)] v_{lx} (1-g_1)\\
		&+[p'_R(V)-p'_R(v_r)] v_{rx} (g_1) + [p'_R(V)-p'_R(V^r)] V_{x}.
	\end{split}&
\end{align}
Note that $(v_i,u_{i}), i=l,r$ is a periodic solution of \eqref{1.1}$_{1,2}$ with the initial data $(\bar{v}_{i}+\phi_{0 i}(x),\bar{u}_{i}+\psi_{0 i}(x))$ and is expected to have the same oscillation as $(v,u)$ near $x= \mp\infty$, respectively.  Thus $(V,U,P )(x,t)$ is expected to have the same oscillation as $(v,u,p)(x,t)$ at the far fields.
\begin{lma}\label{Lem-H}
	Under the assumptions of Lemma \ref{Lem-periodic} , it holds that
	\begin{equation*}
		\|{h_1}\|_{L^1}(t),	\|{h_1}\|_1(t), \|{h_2}\|(t),\|{h_{2t}}\|(t) \leq C \epsilon e^{-\alpha t}, \quad t\geq 0,
	\end{equation*}
	where $ C>0 $ is independent of $ \epsilon $ and $ t, $ and $ \alpha>0 $ is the constant given in   Lemma \ref{Lem-periodic}.
\end{lma}
\subsection{Main Theorem }
We can define   the perturbation   by
\begin{equation*}
	(\phi, \psi, w)(x,0)=(v, u, p)(x,0)-(V, U, P_{R}(V))(x,0).
\end{equation*}
so that $(\phi_{0}, \psi_{0}, w_{0})$ belongs to some Sobolev space. We assume  that the initial data satisfies
\begin{equation}\label{x2017}
	(\phi_{0}, \psi_{0}, w_{0})(x):=(\phi, \psi, w)(x,0)\in H^{1}(\mathbb{R}).
\end{equation}
The main result is
\begin{thm}\label{theorem201}
	Suppose $(\bar{v}_{l},\bar{u}_{l}) $ and $(\bar{v}_{r}, \bar{u}_{r})$ can be connected by $(V^{r},U^{r})$,   (\ref{2.4}) and  (\ref{x2017}) hold, then there exist positive constants $\delta_{0}$, $\epsilon_{0}$ and $\varepsilon_{0}$, such that if $\epsilon< \epsilon_{0}$, $\delta< \delta_0$ and $$\| \left(\phi_{0}, \psi_{0}, w_{0}\right) \|_{1} \leq \varepsilon_{0},$$
then the problem(\ref{1.1})-(\ref{end-behavior})    admits a unique global solution satisfying
\begin{equation}\label{ads2.11}
	\left\{\begin{array}{l}{\left(   v-V ,u-U,p-P_R(V)\right)(x,t)\in C^0([0,+\infty);H^1   )    }, \\
		{\left(   (v-V)_{x} ,(u-U)_{x},(p-P_R(V))_{x}\right)(x,t)\in   L^2([0,+\infty);H^1  )},
	\end{array}\right.
\end{equation}
and
\begin{equation}\label{ads2.12}
	\begin{aligned}
		&\sup_{x \in \mathbb{R}}|  (v-V^r,u-U^r,p-P_{R}(V^r))(x,t)|\rightarrow 0,\quad \forall t\rightarrow \infty.
	\end{aligned}
\end{equation}
\end{thm}

\section{Reformulation of the problem}\label{section3}

To prove this theorem, we define
\begin{eqnarray*}
	(\phi, \psi, w)=(v, u, p)-(V, U, P_{R}(V)).
\end{eqnarray*}
Using (\ref{1.1}) and (\ref{2.11}), one gets that
\begin{align}
	\left\{\begin{array}{l}{\phi_{t}-\psi_{x}=-h_1}, \\ {\psi_{t}+w_{x}=-h_2}, \\
		{w_{t}+E \psi_{x}+w+A(V, \phi)+B(V,U_x)=-p'_{R}(V)h_1},
	\end{array}\right.
	\label{m3.5}
\end{align}
where $A(V, \phi)=\left(p_{R}(V)-p_{R}(V+\phi)\right),B(V,U_x)=\left(E+p_{R}^{\prime}(V)\right) U_{x}$. The initial data   is
\begin{align}\label{m4.6}
	{(\phi, \psi, w)(x, 0)=(v, u, p)(x, 0)-(V, U, P)(x, 0)}.
\end{align}
We will seek the solutions in the functional space $X(0,T)$ for any $0\leq T<\infty$,
\begin{equation}\label{f}
	\begin{aligned} X_\varepsilon(0 ,T) = \{ ( \phi, \psi,w) & \in C \left( 0 , T ; H ^ { 1 } \right) \left|\left( \phi_ { x }, \psi_x, w_x \right) \in L ^ { 2 } \left( 0 , T ; L ^ { 2 } \right) \right. \\
		&\sup _ { 0 \leq t \leq T } \| ( \phi, \psi,w) ( t ) \| _ {}  \leq \varepsilon \}. \end{aligned}
\end{equation}
where the constant $\varepsilon<<1$ is small. 

\begin{pro}[A priori estimate]\label{prop2.1}
	Suppose that  $(\phi,\psi,w) \in X_\varepsilon(0,T)$ is the solution of systems \eqref{2.3}-\eqref{2.6} for some time $T>0$, there exists a positive constant $\varepsilon_0$ independent of $T$ such that if
	\begin{equation}
		\sup\limits_{t\in[0,T]} \|(\phi,\psi,w)(t)\|_{1}\leq\varepsilon\le\varepsilon_0,
	\end{equation}
	then for any $t\in[0,T]$, it holds that
	\begin{equation}\label{2.7}
		\| ( \phi, \psi,w) ( t ) \| _ { 1 } ^ { 2 } + \int _ { 0 } ^ { t } \left\|\left(\phi_{x},  \psi_{x},  w_{x} \right)(\tau)\right\|^{2} d \tau \leq C_0 \left( \left\| \left( \phi_ { 0 } , \psi_ { 0 },w_{0} \right) \right\| _ { 1 } ^ { 2 }+\delta_{0}+\epsilon_{0}\right).
	\end{equation}
	where $C_0>1$ is a constant independent of T.
\end{pro}

Once Proposition \ref{prop2.1} is proved, we can extend the unique local solution $(\phi,\psi)$ to $T=+\infty$. We have the following lemma

\begin{lma}\label{prop201}
	Suppose $(\bar{v}_{l},\bar{u}_{l}) $ and $(\bar{v}_{r}, \bar{u}_{r})$ can be connected by $(V^{r},U^{r})$,   $(\ref{2.4})$ and (\ref{x2017}) hold, then there exist positive constants $\delta_{0}$, $\epsilon_{0}$ and $\varepsilon_{0}$, such that if $\epsilon< \epsilon_{0}$, $\delta< \delta_0$ and $$\| \left(\phi_{0}, \psi_{0}, w_{0}\right) \|_{1} \leq \varepsilon_{0},$$
	then the Cauchy problem (\ref{2.1})(\ref{2.3}) has a unique global solution $( \phi, \psi,w)\in  X_\varepsilon(0 ,\infty)$  satisfying
	\begin{equation*}
		\| ( \phi, \psi,w) ( t ) \| _ { 1 } ^ { 2 } + \int _ { 0 } ^ { \infty } \left\|\left(\phi_{x},  \psi_{x},  w_{x} \right)(\tau)\right\|^{2} d \tau \leq C_0 \left( \left\| \left( \phi_ { 0 } , \psi_ { 0 },w_{0} \right) \right\| _ { 1 } ^ { 2 }+\delta_{0}+\epsilon_{0}\right).
	\end{equation*}
	where $C_0>1$ is a constant independent of T.
\end{lma}


\section{A priori estimate}\label{section4}
Assume that the systems \eqref{m3.5}-\eqref{m4.6} has a solution $( \phi, \psi,w) \in X_\varepsilon(0,T), \varepsilon<<1$ for some $T>0$, that is,
\begin{equation}\label{xx3.1}
	\sup _ { t \in [ 0 , T ] } \| ( \phi, \psi,w) ( t ) \| _ { 1 } \leq \varepsilon,
\end{equation}
then it follows from the Sobolev inequality that there exists  constants  $c_2, d_2$ and $E_1$, satisfying $c_1<c_2<\bar{v}_{l},\bar{v}_{r}<d_2<d_1,$ and
\begin{equation}\label{m3.1}
	|P'_{R}(v) |<E_1<E.
\end{equation}
\begin{lma}\label{lma5.1}
	Under the assumptions of proposition \ref{prop2.1}, it holds that
	\begin{eqnarray}\label{t1}
		\left\|(\phi,\psi,\psi_t,\psi_x )(t)\right\|^{2} +\int_{0}^{t}\left\|  (\psi_x,\psi_t,\sqrt{V_t}\phi )       \right\|^{2}(\tau) d \tau  \leq C(    \|(\phi_{0}, \psi_{0},w_{0} ) \|_{1}^{2} +\delta_{0}^{2}+\epsilon)
	\end{eqnarray}
\end{lma}
\begin{proof}
	At the beginning of this proof, we give two useful equalities which  can be deduced by     \eqref{g2.11} and  \eqref{x2.6} directly.
	\begin{equation}\label{s1}
		V-V^r=(v_l-\bar{v}_l)g_1+ (v_r-\bar{v}_r)(1-g_1)
	\end{equation}
	and
	\begin{equation}\label{s2}
		U^r- {U}=\left(u_{l}-\bar{u}_{l}\right)g_{2}+\left(u_{r}-\bar{u}_{r}\right) \left(1-g_{2}\right),
	\end{equation}
	With the aid of  $(\ref{m3.5})_{2}$ and $(\ref{m3.5})_{3}$, one gets that
	\begin{align}
		\begin{split}
			\psi_{t t}-E \psi_{x x}+\psi_{t}-A(V, \phi)_{x}-B\left(V, U_{x}\right)_{x}=-h_{2t}-h_2+(p'_R(V)h_1)_x.
		\end{split}&
		\label{3.7}\end{align}
We multiply  $(\ref{m3.5})_{1}$ and $(\ref{3.7})$ by $A(V, \phi)$ and $  \mu \psi_{t}+\psi  $, where $\mu=\frac{E_{1}+E}{2 E_{1}}$ , and $ E_{1}=\sup _{v \in[c, d]}\left|p_{R}^{\prime}(v)\right|$, respectively, sum them up, and  intergrading result with  respect to $t$ and $x$ over $ [0,t]\times {\R}$, we have
\begin{align*}
\begin{split}
	&  \frac{1}{2} \int_{\R}  \left( \underbrace{\psi^{2}+ \mu   \psi_{t}^{2}+2 \psi \psi_{t}}_{I_1}   +   \underbrace{\mu E \psi_{x}^{2}+2\mu A \psi_{x}+2\phi(V, \phi) }_{I_2}  \right)  dx \\
	& +\int_{0}^{t}  \int_{\R}  \underbrace{\left(E+\mu\left(p_{R}^{\prime}(V+\phi) \right)\right) \psi_{x}^{2}}_{I_3}  dxd \tau \\
	& +\int_{0}^{t}  \int_{\R} \underbrace{(\mu-1) \psi_{t}^{2}}_{I_4}   dxd \tau  +\int_{0}^{t}  \int_{\R} \underbrace{ V^{r}_{t}\left(P_{R}(V+\phi)-P_{R}(V)-P_{R}^{\prime}(V) \phi\right) }_{I_5}  dxd \tau \\
	=&  \left.    \frac{1}{2} \int_{\R}  \left( \psi^{2}+ \mu   \psi_{t}^{2}+2 \psi \psi_{t}   +    \mu E \psi_{x}^{2}+2\mu A \psi_{x}+2\phi(V, \phi)    \right)  dx\right|_{t=0} \\
	& +\int_{0}^{t}  \int_{\R}  \underbrace{ -\mu  \left[P_{R}^{\prime}(V+\phi)-P_{R}^{\prime}(V)\right] V^{r}_{t} \psi_{x} }_{I_6}   dxd \tau \\
	&+ \int_{0}^{t}  \int_{\R}  \underbrace{   \psi \tilde{B}_{x}}_{I_7}  dxd \tau+\int_{0}^{t}  \int_{\R} \underbrace{\mu \psi_{t} \tilde{B}_{x} }_{I_8}     dxd \tau\\
\end{split}&
\end{align*}
\begin{align*}
\begin{split}&+ \int_{0}^{t}  \int_{\R}  \underbrace{   \psi(B- \tilde{B})_{x}}_{I_9}  dxd \tau+\int_{0}^{t}  \int_{\R} \underbrace{\mu \psi_{t} (B-\tilde{B})_{x} }_{I_{10}}     dxd \tau\\
	&+\int_{0}^{t}  \int_{\R} \underbrace{( V^R-V)_{t}\left(P_{R}(V+\phi)-P_{R}(V)-P_{R}^{\prime}(V) \phi\right) }_{I_{11}}  dxd \tau \\
	& +\int_{0}^{t}  \int_{\R}  \underbrace{ -\mu  \left[P_{R}^{\prime}(V+\phi)-P_{R}^{\prime}(V)\right] (V-V^R)_{t} \psi_{x} }_{I_{12}}   dxd \tau \\
	&+ \int_{0}^{t}  \int_{\R}    \underbrace{(\mu \psi_{t} + \psi)   [   (p'_R(V)h_1)_x-h_{2t}-h_2]-h_1 A}_{I_{13}}    dxd \tau\\
\end{split}&
\end{align*}
where $\tilde{B}=\left[E+p_{R}^{\prime}(V^{r})\right] U^{r}_{x}$,  $\phi(V,\phi):=P_R(V)\phi-\int_{V}^{V+\phi}P_R(s)ds$. First we estimate the left hand side of the above equality. Due to  $\left(\ref{2.4}\right)_{2}  $ and  $\eqref{m3.1}$, there exist positive constants $ c_{i}(i=3, \cdots, 7)$  such that
\begin{equation}
\begin{aligned}
	c_{3}\left(\psi^{2}+\psi_{t}^{2}\right) & \leq I_{1} \leq c_{4}\left(\psi^{2}+\psi_{t}^{2}\right), \quad c_{5}\left(\phi^{2}+\psi_{x}^{2}\right) \leq I_{2}  \leq c_{6}\left(\phi^{2}+\psi_{x}^{2}\right), \\
	I_{3} +I_{4}  & \geq c_{7}\left(\psi_{x}^{2}+\psi_{t}^{2}\right), \quad I_{5}  \geq a_{2} V^{r}_{t} \phi^{2}.
\end{aligned}
\end{equation}
Thus we obtain
\begin{align}\label{m4.26}
\begin{split}
	&   \int_{\R}  \left(  \psi^{2}+     \psi_{t}^{2}    +     \psi_{x}^{2}  +\phi  ^{2}  \right)  dx  +\int_{0}^{t}  \int_{\R}    \psi_{x}^{2} +\psi_{t}^{2}+V^{r}_{t}\phi^{2}  dxd \tau \\
	\leq C &       \int_{\R}  \left(  \psi_{0}^{2}+     \psi_{0t}^{2}    +  \psi_{0x}^{2}  +\phi_{0}  ^{2}  \right)  dx  + C \int_{0}^{t}  \int_{\R} \sum_{6}^{13} I_i    dxd \tau\\
\end{split}&
\end{align}
Next, we estimate the right hand side of (\ref{m4.26})
\begin{align}\label{4.27}
\begin{split}
	\int_{0}^{t}  \int_{\R}  I_6    dxd \tau  & \leq \frac{1}{4}  \int_{0}^{t}  \int_{\R} V^{r}_{t} \phi^{2}dxd \tau +C \int_{0}^{t}  \int_{\R}   V^{r}_{t}    \psi_{x}^{2}   dxd \tau \\
	&\leq \frac{1}{4}  \int_{0}^{t}  \int_{\R} V^{r}_{t} \phi^{2}  dxd \tau  +C \delta_{0} \int_{0}^{t}  \int_{\R} \psi_{x}^{2} dxd \tau ,
\end{split}&
\end{align}
where we have used \eqref{2.14}. Using Holder inequality, one gets
\begin{align}\label{4.28}
\begin{split}
	\int_{0}^{t}  \int_{\R} I_7    dxd \tau \leq & C \int_{0}^{t}  \int_{\R} (|V^{r}_{x}U^{r}_{x}|+|U^{r}_{xx}|) |\psi |   dxd \tau \\
	\leq & C \int_{0}^{t}  (\|V^{r}_{x}U^{r}_{x}\|_{L^1}+\|U^{r}_{xx}\|_{L^1}) \|\psi \|^{\frac{1}{2}}\|\psi _{x}\|^{\frac{1}{2}}  d \tau  \\
	\leq & C \int_{0}^{t}     (\|V^{r}_{x}U^{r}_{x}\|_{L^1}+\|U^{r}_{xx}\|_{L^1})^{\frac{4}{3}}+\|\psi\|^{2}\left\|\psi_{x}\right\|^{2}  d \tau\\
	\leq & C     \left(\delta_{0}^{\frac{4}{3}}+\varepsilon_{0}^{2} \int_{0}^{t}\left\|\psi_{x}\right\|^{2} d \tau\right),
\end{split}&
\end{align}
where we have used the Young's Inequality. Cauchy inequality gives that
\begin{align}\label{4.29}
\begin{split}
	\int_{0}^{t}  \int_{\R}  I_{8}   dxd \tau \leq & \frac{1}{4} \int_{0}^{t}  \int_{\R}        \psi_{t}^{2}                dxd \tau  +C   \int_{0}^{t}  \int_{\R}        B_{x}^{2}                dxd \tau\\
	\leq &\frac{1}{4}  \int_{0}^{t}  \int_{\R}        \psi_{t}^{2}                dxd \tau  +C   \int_{0}^{t}  \int_{\R}        | U^{r}_{xx}|^{2} +      | V^{r}_{x}U^{r}_{x}|^{2}               dxd \tau\\
	\leq &\frac{1}{4}  \int_{0}^{t}  \int_{\R}        \psi_{t}^{2}                dxd \tau  +C  \delta_0^{2}.
\end{split}&
\end{align}
Now, we estimate the error terms $I_9-I_{13}$ term by term. By directly calculation, we have
\begin{align}\label{ppp4.29}
\begin{aligned}
	& B_{x}-\tilde{B}_{x} \\
	= & {\left[E+P_{R}^{\prime}(V)\right]\left(U_{x x}-{{U}^{r}_{x x}}\right)+\left[P_{R}^{\prime}(V)-P_{R}^{\prime}(\tilde{V})\right] {{U}^{r}_{x x}} } \\
	& +P_{R}^{\prime \prime}(V) V_{x}\left(U_{x}-{{U}^{r}_{x}}\right)+P_{R}^{\prime \prime}(V)\left(V_x-{{V}^{r}_{x}} \right){{U}^{r}_{x}} \\
	& +\left[P_{R}^{\prime \prime}(V)-P_{R}^{\prime \prime}(\tilde{V})\right] {{V}^{r}_{x}} {{U}^{r}_{x}} \\
	\leqslant & C\left\{\left|U_{x}-{{U}^{r}_{x}}\right|+\left|U_{x x}-{{U}^{r}_{x x}}\right|+|V-\tilde{V}|+\left|V_{x}-{{V}^{r}_{x}}\right|\right\} \\
	\leqslant & C \sum_{k=0}^{2} \sum_{j=r, l}\left|\frac{\partial}{\partial x^{k}}\left(u_{j}-\bar{u}_{j}\right)\right|+\sum_{k=0}^{2} \sum_{j=r, l}\left|\frac{\partial}{\partial x^{k}}\left(v_{j}-\bar{v}_{j}\right)\right| .
\end{aligned}
\end{align}
Where the last inequality we have used \eqref{s1}  and \eqref{s2}. With the aid of \eqref{ppp4.29}, it follows that
\begin{align}\label{ppp92}
\begin{aligned}
	\int_{0}^{t} \int_{\R} I_{9} d x d t & \leq   \int_{0}^{t}     \|\psi\|  \|(B-\tilde{B}) _x\|  d t \\
	& \leq    \sup_{0\leq \tau\leq t}\|\psi\|(\tau) \int_{0}^{t}       \|(B-\tilde{B}) _x\|     d t   \\
	& \leq   C \varepsilon   \int_{0}^{t}      \epsilon e^{-\alpha t}    d t \leq C \epsilon.
\end{aligned}
\end{align}
We have used Lemma \ref{Lem-periodic}  in the penultimate inequality. Similar like \eqref{4.29}, we have
\begin{align}\label{ppp2}
\begin{aligned}
	\int_{0}^{t} \int_{\R} I_{10} d x d t & \leqslant \frac{1}{4} \int_{0}^{t}\left\|\psi_{t}\right\|^{2} d t+C \int_{0}^{t}\|(B-\tilde{B})_ x\|^{2} d t \\
	& \leqslant \frac{1}{4} \int_{0}^{t}\left\|\psi_{t}\right\|^{2} d t+C \cdot \epsilon^{2}.
\end{aligned}
\end{align}
By directly calculation, one gets that

\begin{align}\label{ppp290}
\begin{aligned}
	\int_{0}^{t} \int_{\R} I_{11} +I_{12} d x d t & \leqslant C \int_{0}^{t}  \int_{\R}    |\left(V-V^{r}\right)_{t} |  ( |\phi^{2}|+|\phi\psi_{x}|) d x      d t \leq  C  \epsilon,
\end{aligned}
\end{align}

where we have used
\begin{align}\label{j1}
\begin{aligned}
	\left(V-V^{r}\right)_{t} & =\left(v_{l}-\bar{v}_{l}\right) g_{1 t}+\left(v_{r}-\bar{v}_{r}\right)\left(1-g_{1}\right)_{t} +v_{l t} g_{1}+v_{r t}\left(1-g_{1}\right) \\
	& \leqslant C\left(\left|v_{l}-\bar{v}_{l}\right|+\left|v_{r}-\bar{v}_{r}\right|+\left|u_{l x}\right|+\left|u_{r x}\right|\right).
\end{aligned}
\end{align}

With the aid of Lemma \ref{Lem-H},  using Holder inequality, it follows that

\begin{align}\label{x4.29}
\begin{split}
	\int_{0}^{t} \int_{\R} I_{13} d x d t & \leqslant    C \int_{0}^{t}     \int_{\R}   (|\phi|+|\psi_{t}|+|\psi|) (|h_1|+ |h_{1x}|+ |h_2|+ |h_{2t}|)   d x d t \\
	& \leqslant   \int_{0}^{t}       (\|\phi\|+\|\psi_{t}\|+\|\psi\|) (\|h_1\|+ \|h_{1x}\|+ \|h_2\|+ \|h_{2t}\|)                                          d t \leqslant   \epsilon .
\end{split}&
\end{align}






Using (\ref{m4.26})-(\ref{4.29}), \eqref{ppp92}-\eqref{ppp290} and (\ref{x4.29}), taking $\varepsilon_{0}$ sufficiently small, we can obtain (\ref{t1}) immediately. Lemma \ref{lma5.1} is completed.
\end{proof}

\begin{lma}\label{lma5.2}
Under the assumptions of proposition \ref{prop2.1}, it holds that
\begin{eqnarray}\label{4.41}
\left\|(w,w_x,w_{t})(t)\right\|^{2}(t)+\int_{0}^{t}\left\|(w,w_{t})\right\|^{2}(\tau) d \tau \leq  C \left( \|(\phi, \psi, w)(0)\|_{1}^{2}+\delta_{0}+\epsilon\right)
\end{eqnarray}
\end{lma}

\begin{proof}
By (\ref{m3.5}), we see that $w_x=-\psi_t-h_2$, thus the estimate of $w_x$ in  (\ref{4.41})  comes from Lemma  \ref{lma5.1} and  Lemma \ref{Lem-H} directly.  Differentiating $(\ref{m3.5}) _3$ with respect to $t$, then multiplying the result by $w_{t}$, we obtain that

\begin{align}\label{4.44}
\begin{split}
&\frac{1}{2}\left( w_{t}^{2}+  E w_{x}^{2}\right)_{t}-E\left(w_{x} w_{t}\right)_{x}+w_{t}^{2}=-\left(A   +B+p'_{R}(V)h_1\right)_{t} w_{t}.
\end{split}&
\end{align}
Integrating (\ref{4.44}) over $[0,t] \times (- \infty ,\infty)$, one get
\begin{align}\label{4.45}
\begin{split}
&   \int_{\R}   \left(\frac{1}{2} w_{t}^{2}+\frac{1}{2} E w_{x}^{2}\right)_{} dx          +    \int_{0}^{t} \int_{\R} w_{t}^{2} dxd\tau \\
=& \int_{\R}   \left(\frac{1}{2} w_{t}^{2}+\frac{1}{2} E w_{x}^{2}\right)\Big|_{t=0}dx+ \int_{0}^{t} \int_{\R} \underbrace{-\tilde{M} w_{t}}_{I_{14}} dxd\tau       +\int_{0}^{t} \int_{\R} \underbrace{-\tilde{B}_{t} w_{t}}_{I_{15}} dxd\tau\\
&+\int_{0}^{t} \int_{\R} \underbrace{      [p'_{R}(V)h_1 ]_{t} w_{t}     }_{I_{16}} dxd\tau+ \int_{0}^{t} \int_{\R} \underbrace{(A_{t}-\tilde{M}) w_{t}}_{I_{17}} dxd\tau       +\int_{0}^{t} \int_{\R} \underbrace{(B-\tilde{B})_{t} w_{t}}_{I_{18}} dxd\tau,
\end{split}&
\end{align}
where $\tilde{M}=\left[P_{R}^{\prime}(V)-P_{R}^{\prime}(V+\phi)\right] V^{r}_{t}-P_{R}^{\prime}(V+\phi) \psi_{x}$. We estimate right hand side  below.
Using $\eqref{2.11}_{2} $ and Lemma {\ref{xingzhi}}, we get
\begin{align*}
\tilde{ B}_{t}\leq    C   ( |V^{r}_{xx}| +  (V^{r}_{x})^{2} ).
\end{align*}
Cauchy inequality gives that
\begin{align}
\begin{split}
\int_{0}^{t} \int_{\R} (I_{14}+I_{15}) d x d \tau \leq&\frac{1}{4}  \int_{0}^{t} \int_{\R} w_t^{2} dxd\tau   +C  \int_{0}^{t} \int_{\R}  \left(\tilde{M} ^{2} + {\tilde{B}_{t}} ^{2} \right) dx d\tau \\
\leq&\frac{1}{4}  \int_{0}^{t} \int_{\R} w_t^{2} dxd\tau   +C \left( \int_{0}^{t} \int_{\R}  \left(\psi_{x}^{2} +(\sqrt{V^{r}_{t}} \phi)^{2} \right) dx d\tau+\delta_{0}\right).
\end{split}&
\end{align}



Using the same method in (\ref{x4.29}), the error terms $I_{16}$ can be estimated as
\begin{align}
\begin{split}
\int_{0}^{t}  \int_{\R}   I_{16}    dxd \tau\leq C \epsilon.
\end{split}&
\end{align}
Using (\ref{j1}) and Lemma \ref{Lem-H}, one gets that
\begin{align}\label{b4029}
\begin{split}
\int_{0}^{t}  \int_{\R}   I_{17}    dxd \tau=&   \int_{0}^{t}  \int_{\R}  \left\{  \left(p'_{R}(V)-p'_{R}(V+\phi)\right)(V_t-V^r_t)+p'_{R}(V+\phi) h_1  \right\}  w_{t} dx d\tau\\
\leq&\frac{1}{4}  \int_{0}^{t} \int_{\R} w_t^{2} dxd\tau   +C  \int_{0}^{t} \int_{\R}  |\phi|^2   |  V_t-V^r_t|^2 + |h_1|^2dx d\tau\\
\leq&\frac{1}{4}  \int_{0}^{t} \int_{\R} w_t^{2} dxd\tau   +C \sup_t \|\phi\|^2 \int_{0}^{t} \int_{\R}    |  V_t-V^r_t|^2 dx d\tau +C  \int_{0}^{t}  \|h_1\|^2 d\tau\\
\leq&\frac{1}{4}  \int_{0}^{t} \int_{\R} w_t^{2} dxd\tau   +C \epsilon,
\end{split}&
\end{align}
where we have used \eqref{xx3.1}, \eqref{s1}, Lemma \ref{Lem-periodic} and Lemma \ref{Lem-H}. By directly  calculation, one gets that
\begin{align*}
\begin{aligned}
& B_{t}-\tilde{B}_{t} \\
\leqslant & C\left\{\left|U_{x}-{{U}^{r}_{x}}\right|+\left|U_{x t}-{{U}^{r}_{x t}}\right|+|V-\tilde{V}|+\left|V_{t}-{{V}^{r}_{t}}\right|\right\} \\
\leqslant & C \sum_{k=0}^{2} \sum_{j=r, l}\left|\frac{\partial}{\partial x^{k}}\left(u_{j}-\bar{u}_{j}\right)\right|+\sum_{k=0}^{2} \sum_{j=r, l}\left|\frac{\partial}{\partial x^{k}}\left(v_{j}-\bar{v}_{j}\right)\right| .
\end{aligned}
\end{align*}
Thus similar like  \eqref{b4029}, one  gets that
\begin{align}\label{jidan}
\begin{split}
\int_{0}^{t}  \int_{\R}   I_{18}    dxd \tau \leq\frac{1}{4}  \int_{0}^{t} \int_{\R} w_t^{2} dxd\tau   +C \epsilon,
\end{split}&
\end{align}
Using (\ref{4.45})-(\ref{jidan}), we can obtain (\ref{4.41}) immediately. Lemma \ref{lma5.2} is completed.
\end{proof}
Proposition \ref{prop2.1} is directly proved from Lemma \ref{lma5.1}-\ref{lma5.2}.
\section{The proof of theorem \ref{theorem201} }\label{section5}
Now, we turn to proof of the main theorem, i.e, Theorem \ref{theorem201}. With the help of Lemma \ref{prop201},  one can obtain (\ref{ads2.11}) immediately. It remain to show    (\ref{ads2.12}). We will use the following lemma.
\begin{lma}[\cite{mn1985}]\label{lemma5.1}
Suppose that the function \ $f(t) \geq 0\in L^1(0, +\infty) \cap  BV(0, +\infty) $. Then  it holds that   $t \rightarrow \infty$  as  $f(t) \rightarrow0$.
\end{lma}
\begin{proof}  [  proof of theorem \ref{theorem201}]
Differentiating  $(\ref{m3.5})_1$ with   respect to $x$, multiplying the resulting  by  $\phi_{x}$, and integrating on  $(-\infty,\infty)$, we have
\begin{equation*}
\left|\frac{{ d}}{{ d}t}\left(\|\phi_{x}\|{}^{2}\right)\right|\leq C(\|\phi_{x}\|{}^{2} +\|\psi_{xx}\|{}^{2}  +\|h_{1x}\|{}^{2}    )  .
\end{equation*}
Using  Lemma \ref{prop201} and Lemma {\ref{Lem-H}}, we have
\begin{equation*}
\int_{0}^{\infty} \left|\frac{{ d}}{{ d}t}\left(\|\phi_{x}\|{}^{2}\right)\right| { dt} \leq C_0 \left( \left\| \left( \phi_ { 0 } , \psi_ { 0 },w_{0} \right) \right\| _ { 1 } ^ { 2 }+\delta_{0}+\epsilon_{0}\right)\leq C.
\end{equation*}
Thus  one gets that  $\|\phi_{ x}\|{}^{2}\in L^1(0, +\infty) \cap  BV(0, +\infty)$. By Lemma  \ref{lemma5.1},
\begin{equation*}
\|\phi_{ x}\|{}\rightarrow0, \quad \text{as} \quad t\rightarrow+\infty.
\end{equation*}
Since  $\|\phi_{xx}\|{}$ is bounded, the  Sobolev inequality gives that
\begin{eqnarray*}
\|{}{v}-V\|_{L^{\infty}}^{2}=\|\phi_{x}\|_{L^{\infty}}^{2} \leq 2\| \phi_{x}(t)\|{}  \| \phi_{xx}(t)\|{} \rightarrow  0.
\end{eqnarray*}
With the aid of \eqref{s1}, \eqref{s2} and Lemma \ref{Lem-periodic}, one gets that
\begin{align*}
\begin{split}
\lim_{t\rightarrow \infty}\sup_{x\in \R}|{v}-V^{r}|\leq \lim_{t\rightarrow \infty}\sup_{x\in \R}(|{v}^{r}-V|+|V-V^{r}|)=0.
\end{split}&
\end{align*}
Thus $\lim_{t\rightarrow \infty}\sup_{x\in \R}|{v}-V^{r}|=0$,  similarly, we can prove that
\begin{align*}
\begin{split}
\lim_{t\rightarrow \infty}\sup_{x\in \R}|({u}-U^{r}, p-P_{R}(V^r))|=0.
\end{split}&
\end{align*}
Thus    theorem \ref{theorem201} is proved.
\end{proof}
\section{Appendix}
\subsection{Proof of Lemma \ref{Lem-H}}
\begin{proof}
1) First, we estimate $h_1$. With the help of \eqref{x2.11}, by directly calculation, one obtain that
\begin{align*}
\begin{split}
h_1^{2}  \leq   \sum_{j=r,l}&\left( \left|v_j-\bar{v}_j\right|^{2}+\left|u_j-\bar{u}_j \right|^{2}+ \left| u_{jx}\right|^{2} \right)  (\left|g_{1t}\right|^{2}+\left|g_{2t}\right|^{2}+\left|g_1-g_2\right|^{2}).
\end{split}&
\end{align*}
Using Holder inequality and Lemma \ref{Lem-periodic}, one gets that
\begin{align}\label{y2.12}
\begin{split}
\|h_1 \| \leq   & \|(v_r-\bar{v}_r)^{2},(\bar{v}_l-v_l)^{2},(u_r-\bar{u}_r)^{2},(\bar{u}_l-u_l)^{},u_{rx }^{2},u_{lx }^{2}\|_{L^1}^{\frac{1}{2}}   \|g_{1t}^{2},g_{2t}^{2},g_1^{2},g_2^{2}\|_{L^{\infty}}^{\frac{1}{2}}\\
=  & \|(v_r-\bar{v}_r),(\bar{v}_l-v_l),(u_r-\bar{u}_r),(\bar{u}_l-u_l),u_{rx },u_{lx }\|    \|g_{1t},g_{2t},g_1,g_2\|_{L^{\infty}}\\
\leq &C\epsilon e^{-\alpha t}   \|g_{1t},g_{2t}, g_1,g_2 \|_{L^{\infty}}\leq C \epsilon e^{-\alpha t}.
\end{split}&
\end{align}
The last inequality we have used \eqref{2.14} and  \eqref{2.15}.

2) This part we estimate the term  $h_{1x}$. Using
\begin{align*}
\begin{split}
h_{1x}=& (v_r -v_l)_{x}g_{1t}-(u_r -u_l)_{x}g_{2t}+(u_r-u_l)_{x}(g_1-g_2)_{x}\\
&+(u_r-u_l)_{xx}(g_1-g_2)+(v_r-\bar{v}_r+\bar{v}_l-v_l) g_{1tx}-(u_r-\bar{u}_r+\bar{u}_l-u_l) g_{2tx},\\
\end{split}&
\end{align*}
we obtain
\begin{align*}
\begin{split}
h_{1x}^{2}\leq & \sum_{j=r,l}\left( \left|v_j-\bar{v}_j\right|^{2}+\left|u_j-\bar{u}_j \right|^{2}  + \left| u_{jx}\right|^{2} + \left| u_{jxx}\right|^{2}\right) \sum_{i=1,2}(\left|g_{it}\right|^{2} +\left|g_i\right| ^{2}  +\left|g_{itx}\right|^{2} +\left|g_{ix}\right| ^{2}     )\\
\leq &C \sum_{j=r,l}\left( \left|v_j-\bar{v}_j\right|^{2}+\left|u_j-\bar{u}_j \right|^{2}  + \left| u_{jx}\right|^{2} + \left| u_{jxx}\right|^{2}\right) \sum_{i=1,2}(   \left|g_i\right|^{2}   +\left|g_{ixx}\right|^{2} +\left|g_{ix}\right|^{2}  +  {g_{1x} ^{4}}   ).\\
\end{split}&
\end{align*}
The last inequality we have used \eqref{m2.11} and \eqref{g2.11}. Similar like \eqref{y2.12}, we obtain that
\begin{equation*}
\|h_{1x}\|   \leq C \epsilon e^{-\alpha t}.
\end{equation*}

3) Using$(\ref{2.4})_4$ and \eqref{x2.6},   the term  $h_2$  can be estimated as
\begin{align*}
\begin{split}
h_2^{2}\leq & \sum_{j=r,l}\left( \left|v_j-\bar{v}_j\right|^{2}+\left|V-V^r \right|^{2}  + \left| v_{jx}\right|^{2}  \right)  (\left|g_{2x}\right|^{2} +\left| a_{2 }(V-v_l)(1-g_1)\right| ^{2}  +\left| a_{2 }(V-v_r) g_{1}\right|^{2} +\left| V_{x}\right|  ^{2}    )\\
\leq & \sum_{j=r,l}\left( \left|v_j-\bar{v}_j\right|^{2}  + \left| v_{jx}\right|^{2}  \right)  (\left|g_{2x}\right|^{2} +\left| a_{2 }(V-v_l)(1-g_1)\right| ^{2}  +\left| a_{2 }(V-v_r) g_{1}\right|^{2} +\left| V_{x}\right|  ^{2}    )\\
\leq &C \sum_{j=r,l}\left( \left|v_j-\bar{v}_j\right|^{2}  + \left| v_{jx}\right|^{2}  \right)  (\left|g_{2x}\right|^{2} +\left|  (1-g_1)\right| ^{2}  +\left|  g_{1}\right|^{2} +\left| V_{x}\right|  ^{2}    ).\\
\end{split}&
\end{align*}
where we have used \eqref{s1} in the second inequality and $(\ref{2.4})_2$, (\ref{2.14}), (\ref{decay-per}), (\ref{x2.6})in the last inequality.

4) Since $h_{2t}$ is very complex,  we estimate this term  carefully.
\begin{align}\label{hbu2.11}
\begin{split}
h_{2t}= &[p'_R(V)-p'_R(v_l)]_{t} v_{lx} (1-g_1)-[p'_R(V)-p'_R(v_l)] v_{lx}  g_{1t} \\
&+[p'_R(V)-p'_R(v_l)] v_{lxt} (1-g_1)+[p'_R(V)-p'_R(v_r)] u_{rxx} g_1 \\
&+[p'_R(V)-p'_R(v_r)] v_{rx} g_{1t} + [p'_R(V)-p'_R(v_r)]_{t} v_{rx} g_1\\
&+(v_r-\bar{v}_r+\bar{v}_l-v_l)g_{2xt} + [p'_R(V)-p'_R(V^r)] V_{xt}\\
&+\underbrace{[p'_R(V)-p'_R(V^r)]_{t} V_{x}}_{W_{1}}+\underbrace{(v_r -v_l)_{t}g_{2x}}_{W_{2}}.
\end{split}&
\end{align}
Since every term except $W_1$, $W_2$ in \eqref{hbu2.11} can be controlled by  $\frac{\partial^{k}}{\partial x^{k}} \{(v_{j} - \bar{v}_{j}, u_{j}-\bar{u}_{j})\}, k=0,1,2; j=r,l$, we concentrate the main effort on $W_1$ and $W_2$. Using
$$  u_{j t}=-p_{R}\left(v_{j}\right)_{x} =-p_{R}^{\prime}\left(v_{j}\right)v_{jx}  \quad j=\gamma, l, $$
one gets that
\begin{equation*}
u_{r t}-u_{l t}=p^{\prime}\left(v_{l}\right)     v_{l x}   -p^{\prime}\left(v_{r}\right) v_{r x} \leq  C\left(\left|v_{lx}\right|+\left|v_{r x}\right|\right),
\end{equation*}
where we have used $(\ref{2.4})_3$ and \eqref{Lem-periodic}. Using \eqref{s1}  and \eqref{s2}, it follows that
\begin{align}\label{s34}
\begin{split}
&\left[ p_{R}^{\prime}(V)-p_{R}^{\prime}(V^r)\right]_{t}\\
=&p_{R}^{\prime \prime}(V) V_{t}-p_{R}^{\prime \prime}(V) V^r_{t}+p_{R}^{\prime \prime}(V) V^r_{t}-p_{R}^{\prime \prime}(V^r) V^r_{t} \\
\leq &C\left\{(U^r-U)_{x}+h_1+\left[p^{\prime \prime}(V)-\rho^{\prime \prime}(V^r)\right]\right\} \\
\leq& C \sum_{j=r,l} \left( \left|v_j-\bar{v}_j\right| +\left|u_j-\bar{u}_j \right| + \left| u_{jx}\right|  \right)  (\left|g_{1x}\right| +\left|g_{2x}\right| +\left|g_1\right|+\left|g_2\right| )+h_1.
\end{split}
\end{align}
Combination \eqref{y2.12}-\eqref{s34}, we have
\begin{equation*}
\|h_{2t}\|   \leq C \epsilon e^{-\alpha t}.
\end{equation*}
\end{proof}


\begin{thebibliography}{aa}
\bibitem{HP1999}L. Hsiao, R. Pan,  Nonlinear stability of rarefaction waves for a rate-type viscoelastic system. Chinese Ann. Math. Ser. B20(1999), no.2, 223-C232.
\bibitem{HLWX} M. Hou, L. Liu, L. Xu, S. Wang
Vanishing viscosity limit to the planar rarefaction wave with vacuum for 3-D full compressible Navier-Stokes equations with temperature-dependent transport coefficients. arXiv:2308.03156.

\bibitem{HT}L. Hsiao, R. Pan, Nonlinear stability of two-mode shock profiles for a rate-type viscoelastic system with relaxation. Chinese Ann. Math. Ser. B20(1999), no.4, 47-C488.
\bibitem{hm2009} F. Huang, A. Matsumura, Stability of a composite wave of two viscous shock waves for full compressible Navier-Stokes equation. Comm. Math. Phys. 289 (2009), no. 3, 841-861.
\bibitem{HMX}F. M. Huang, A. Matsumura and Z. P. Xin,
Stability of contact discontinuities for the 1-D compressible Navier-Stokes equations.
Arch. Ration. Mech. Anal. 179, No. 1, 55-77 (2006).
\bibitem{HXXY}F. Huang, Z. Xin, L. Xu, Q. Yuan, Nonlinear asymptotic stability of compressible vortex sheets with viscosity effects. arXiv:2308.06180.
\bibitem{HX}F. Huang, L.  Xu, Decay rate toward the traveling wave for scalar viscous conservation law. Commun. Math. Anal. Appl. 1(2022), No. 3, 395-409 .
\bibitem{HXY2022} F. Huang,  L. Xu,  Q. Yuan,   Asymptotic stability of planar rarefaction waves under periodic perturbations for 3-d Navier-Stokes equations. Adv. Math. 404 (2022), Paper No. 108452, 27 pp.
\bibitem{HY2021} F. Huang, Q. Yuan, Stability of large-amplitude viscous shock under periodic perturbation for 1-d isentropic Navier-Stokes equations, Comm. Math. Phys.387(2021), no.3, 1655-C1679.
\bibitem{6}H. Li,   R.Pan,  Zero relaxation limit for piecewise smooth solutions to a rate-type viscoelastic system in the presence of shocks. J. Math. Anal. Appl.252(2000), no.1, 298-C324.
\bibitem{10}T. Luo, D. Serre,  Linear stability of shock profles for a ratetype viscoelastic system with relaxation. Quart. Appl. Math. 56 (1998), no. 3, 569-586.
\bibitem{mn1985}  A. Matsumura, K.Nishihara, On the stability of traveling wave solutions of a one-dimensional model system for compressible viscous gas. Jpn. J. Appl. Math, 1985, V2 (1): 17-25.
\bibitem{MN1986} A. Matsumura and K. Nishihara, Asymptotics toward the rarefaction wave of the solutions of a one-dimensional model system for compressible viscous gas, Japan J. Appl. Math., 3 (1986), $1-13$.
\bibitem{SI1990}I. Suliciu, On modelling phase transitions by means of rate-type constitutive equations, shock wave
structure, Int J. Engng. Sci., 28(1990), 827-841.

\bibitem{LWX} L. Liu, D.  Wang, L. Xu, Asymptotic stability of the combination of a viscous contact wave with two rarefaction waves for 1-D Navier-Stokes equations under periodic perturbations. J. Differ. Equations (2023)346, 254-276 .
\bibitem{LWX2} L. Liu, S.  Wang, L. Xu, Decay rate to the planar viscous shock wave for multi-dimensional scalar conservation laws. arXiv:2312.03553.

\bibitem{LWX3} L. Liu, S.  Wang, L. Xu,
Optimal decay rates to the contact wave for 1-D compressible Navier-Stokes equations. arXiv:2310.12747.
{\bibitem{Liu1} T. P. Liu, Nonlinear stability of shock waves for viscous conservation laws, Mem. Amer. Math. Soc., 56 (1985), 1-108.}
{\bibitem{Liu2} T. P. Liu and Z. P. Xin, Pointwise decay to contact discontinuities for systems of viscous conservation laws, Asian J. Math., 1 (1997), 34-84.}
\bibitem{Liu3}T. P. Liu and Y. N. Zeng, Shock waves in conservation laws with physical viscosity, Mem. Amer. Math. Soc., 234 (2015), no. 1105, vi+168.
\end{thebibliography}
\end{document}